\documentclass[11pt]{article}

\usepackage{amsmath,amssymb,amsthm}
\usepackage{graphicx}

\usepackage{geometry,graphicx,xcolor,booktabs,enumitem}
\usepackage[colorlinks=true,linkcolor=blue!60!black,citecolor=blue!60!black]{hyperref}
\usepackage{tikz}
\usetikzlibrary{arrows.meta,positioning}
\usepackage{xcolor}
\usepackage{color,soul}

\definecolor{myblue}{RGB}{90,140,200}

\definecolor{orange}{rgb}{0.8, .2, 0.0}

\definecolor{cR}{RGB}{33,102,172}
\definecolor{cS}{RGB}{214,96,77}
\definecolor{cA}{RGB}{26,150,65}
\definecolor{cB}{RGB}{178,24,43}

\theoremstyle{definition}
\newtheorem{definition}{Definition}[section]
\def\bR{{\mathbb R}}
\theoremstyle{plain}
\newtheorem{lemma}[definition]{Lemma}
\newtheorem{proposition}[definition]{Proposition}
\newtheorem{thmX}{Theorem}
 
\newtheorem{theorem}{Theorem}
\newtheorem{corollary}[definition]{Corollary}

\theoremstyle{remark}

\theoremstyle{conjecture}

\newcommand{\eps}{\varepsilon}
\newcommand{\KX}{\cK(X)}
\newcommand{\dH}{d_H}

\def\cC{{\cal C}}

\def\cF{{\cal F}}
\def\cG{{\cal G}}
\def\cK{{\cal K}}

\def\cI{{\cal I}}

\def\cO{{\cal O}}

\def\cR{{\cal R}}

\def\Gl{\cG_\cF}
\def\epsilon{\varepsilon}
\newcommand{\BF}{\bf\boldmath }

\def\Hto{\succcurlyeq_{H_\cF}}
\def\cC{{\cal C_{\cal F}}}
\def\Cto{\succcurlyeq_{\cC}}

\usepackage{lineno}
\newcommand{\beq}{\begin{linenomath}\begin{equation*}} 
\newcommand{\eeq}{\end{equation*}\end{linenomath}} 
\newcommand{\beqn}{\begin{linenomath}\begin{equation}} 
\newcommand{\eeqn}{\end{equation}\end{linenomath}} 
\def\Om{\Omega_\cF}
\def\P{\Phi_\cF^I}

\def\Pn{\Phi_\cF^{I_n}}

\def\Ceq{\overset{\cC}{=}}

\DeclareMathOperator{\Node}{Node}

\title{The Chain Graph of a general Iterated Function System}
\author{Roberto De Leo}
\date{\today}

\begin{document}
\maketitle

\begin{abstract}
    Classical fractal geometry describes the metric and topological structure of attractors generated by contractive Iterated Function Systems (IFSs). 
    Much less is known about their global qualitative dynamics once the contractive hypothesis is abandoned. 
    In this article we focus on IFSs with ``compact dynamics'', namely those IFSs for which there exists a non-empty compact set invariant under the Hutchinson map  that attracts every compact subset of the phase space. 
    We call such a set the ``global attractor'' of the IFS.
    We introduce the ``chain graph'' of a general IFS as a directed graph encoding the qualitative dynamics features of the IFS.
    For an IFS with compact dynamics, the chain graph contains at least one node. 
    Our main results are that the chain graph of an IFS with compact dynamics coincides with the chain graph of its restriction to its global attractor and that the chain graph of an IFS has at most as many nodes as the graph of its Hutchinson map.
\end{abstract}

\section{Introduction}

The main goal of the present article is to extend to Iterated Function Systems (IFSs) some ideas recently introduced by James Yorke and the present author in the study of the qualitative behavior of semiflows.

Historically, IFSs have been used quite successfully for the construction of deterministic fractals~\cite{Bar88} in the context of contractive maps, beginning with the seminal article of Hutchinson~\cite{Hut81} and the introduction of the general theory of contractive IFS by Michael Barnsley in~\cite{Bar88}. 
Recently, Barnsley and Andrew Vince~\cite{BV12,Vin13} and, independently, the present author~\cite{DL08,DL15,DL15b} started investigating some self-projective IFSs on real and complex projective spaces whose maps are not contractions.
At the same time, Barnsley et al. started investigating properties of ``general'' IFSs~\cite{BV11,BV13,BL14}, namely IFSs where the only requirement is that maps be continuous, mostly from the point of view of the {\em chaos game}.
In this article, we embrace precisely this point of view but our aim is somewhat complementary: we consider general IFSs as a dynamical system and introduce a tool, the chain graph, to study and represent the qualitative behavior of such systems. 

For simplicity, we restrict our attention to locally compact metric spaces. 
We do not require completeness and we assume instead that our IFSs have ``compact dynamics''.
More precisely, we require our IFSs to have a non-empty compact $H$-invariant (i.e. invariant under the Hutchinson map) set that attracts every compact subset of the phase space. 
Such a set is unique and we call it the global attractor.
We showed in~\cite{DLY26a} that the qualitative behavior of a semiflow with compact dynamics is equivalent to that of its restriction to its global attractor, and here we show that the same happens in case of IFSs.
This way, we switch the focus from the topological properties of the phase space to the dynamical properties of the IFS.

Given an IFS $\cF$ with compact dynamics on a locally compact metric space $X$, we use an adaptation of Charles Conley's ``$\eps$-chains'' to decompose the dynamics of $\cF$ into its ``recurrent'' and ``gradient'' components and represent the qualitative dynamics of $\cF$ in a graph. 

Our main results are the following.
Let $\cF$ be an IFS.
Then: 
1. $\cF$ has compact dynamics if and only if it has a compact global trapping region (Theorem~\ref{thm: compact dynamics = compact gl tr reg});
2. if $\cF$ has compact dynamics, then the chain-recurrent set, chain nodes and chain edges coincide with those of the restriction of $\cF$ to its global attractor (Theorem~\ref{prop: C = C glob}) and each chain node $N$ of $\cF$ is forward-invariant and its omega-limit is a non-empty, compact and $H$-invariant subset of $N$ (Theorem~\ref{thm: N is H-invariant});
3. $\cF$ has compact dynamics if and only if its Hutchinson map has compact dynamics (Theorem~\ref{thm: F comp dyn iff H comp dyn}) and the chain graph of its Hutchinson map has at least as many nodes as the chain graph of $\cF$ (Theorem~\ref{thm: H one node => F one node}).
As a concrete application, we show that the chain graphs of the Levitt-Yoccoz gasket and of its Hutchinson map have a single node and no edge (Theorem~\ref{thm: LY}). 

    \medskip
    The article is structured as follows.
    In Section~\ref{sec: IFS} we introduce the main definitions about the dynamics of an IFS.
    In Section~\ref{sec: global attractor} we introduce global attractors and trapping regions and prove their main properties. 
    In Section~\ref{sec: chains} we use chains to define the concept of ``being downstream'', with which we define chain nodes, chain edges and the chain graph.
    In Section~\ref{sec: Hutch}, we compare chain nodes and chain graph of an IFS with those of the corresponding Hutchinson map.
    Finally, in Section~\ref{sec: cubic gasket} we consider an important example of a non-contractive IFS, the Levitt-Yoccoz gasket, and prove that, just as in case of contractive IFSs, it has a single node.   
\section{Iterated Function Systems}
\label{sec: IFS}
We introduce several definitions and some related basic results on the qualitative dynamics of IFSs.
Our goal is to be able to describe the qualitative behavior of general IFSs, namely IFSs whose maps are not necessarily contractions.

\medskip\noindent
{\bf The phase space.} Throughout the article, {\BF$X$} will denote a {\bf metrizable locally compact} topological space. 
We will usually denote points in $X$ by $x,y,z$ and $d(x,y)$ will denote the distance between $x$ and $y$ for some metric $d$ compatible with the topology of $X$. 

\medskip\noindent
{\bf The multi-indices semigroup.} We denote by {\BF$\cI^m$} the free monoid generated by $\{1,\dots,m\}$ with identity 0.
Given $I=i_1\dots i_k\in\cI^m$, we denote by {\BF$|I|=k$} the number of indices in $I$ and call this number the {\BF length} of $I$.
We denote by {\BF $\cI^m_k$} the set of all indices of $\cI^m$ of length at least $k$.


\medskip\noindent
{\bf Iterated Function Systems.} 
%
    An {\bf iterated function system} (IFS) on a topological space $X$ is a finite collection
    $\cF = \{f_1,\dots,f_m\},$
    of continuous maps of $X$ into itself.
    Given $I=i_1\dots i_k\in\cI^m$, we set
    $
    \P = f_{i_1}\circ\dots\circ f_{i_k}.
    $
    An IFS $\cF$ determines a continuous semigroup action
    $
     \Phi_\cF: \cI^m \times X \to X
    $
    by
    $
     \Phi^I_\cF(x) = f_{i_1}\circ\dots\circ f_{i_k}(x).
    $
    In particular, this means that:
    \begin{enumerate}
        \item $\Phi_\cF^0=id_X$;
        \item $\Phi^I_\cF\circ\Phi^J_\cF=\Phi^{I\cdot J}_\cF$ for every $I,J\in\cI^m$.
    \end{enumerate}
    We call $\Phi_\cF$ the {\bf semiflow} associated to $\cF$.
    Notice that, when $m=1$, this reduces to the definition of {\em discrete-time semiflow} on $X$.
    Finally, we say that $\cF$ is {\bf contractive} when each of its maps $f_i$ is a contraction.
 
\medskip\noindent
{\bf Invariance under $\cF$.}
  We say that a set $A\subset X$ is {\bf forward-invariant} (resp. {\bf backward-invariant}) under $\cF$ if $\Phi^I_\cF(A)\subset A$ (resp. $\Phi^I_\cF(A)\supset A$) for every $I\in\cI^m$.
  Equivalently, if $f_i(A)\subset A$ (resp. $f_i(A)\supset A$) for every $i=1,\dots,m$.
  We say that a set $A$ is {\bf invariant} under $\cF$ if it is both backward- and forward-invariant, namely if $\Phi^I_\cF(A)=A$ for every $I\in\cI^m$.
  Finally, we write {\BF$H_\cF(A)=\cup_{i=1}^m f_i(A)$} and say that $A$ is {\bf $H$-invariant} under $\cF$ if $H_\cF(A)=A$.

\medskip\noindent
{\bf Limit sets and recurrence.}
    We denote by $\text{\BF$\cO_\cF(x)$}=\{\P(x):I\in\cI^m\}$ the orbit of $x\in X$ under $\cF$.
    The limit set of a point $x\in X$ under $\cF$ is the set
    \beq
    \text{\BF$\Om(x)$} = \{y: \text{ there are }\; I_n\in\cI^m\text{ such that }|I_n| \to\infty\text{ and }\Phi^{I_n}_\cF(x)\to y\}.
    \eeq
    Given a subset $A\subset X$, we define
    \beq
    \text{\BF$\Om(A)$} = \{y:\text{ there are }I_n\in\cI^m\text{ and }a_n\in A\text{ such that }|I_n|\to\infty\text{ and }\Phi^{I_n}_\cF(a_n)\to y\}.
    \eeq
    Notice that, for a singleton $A=\{x\}$, the definition reduces to the limit set of a point, so $\Om(\{x\})=\Om(x)$.
    We say that $x$ is {\bf fixed} for $\cF$, or that $x$ is a {\bf fixed-point} of $\cF$, if $\cO_\cF(x)=\{x\}$.
    Finally, we say that $x$ is {\bf recurrent} under $\cF$ if, for every $z\in\cO_\cF(x)$,
    $x\in\Om(z)$ and we denote by {\BF$\cR_\cF$} the set of all points that are recurrent under $\cF$.

    \medskip
    Notice that this definition of recurrence reduces to the standard one for semiflows when $m=1$.

\begin{proposition}
    \label{prop: x in R => x in Om(x)}
    Let $x\in\cR_\cF$. 
    Then $x\in\Omega_\cF(x)$.
\end{proposition}

\begin{proof}
    Let $z\in\cO_\cF(x)$.
    Since $x\in\cR_\cF$, there are $J_n\in\cI^m$ with $|J_n|\to\infty$ such that $\Phi_\cF^{J_n}(z)\to x$.
    Let $I\in\cI^m$ such that $z=\P(x)$.
    Then $\Phi_\cF^{J_nI}(x)\to x$ with $|J_nI|>|J_n|\to\infty$, namely $x\in\Om(x)$.
\end{proof}

It is well-known that, when $m=1$, the inverse of Proposition~\ref{prop: x in R => x in Om(x)} holds as well. 
For $m>1$, it does not.
For example, just consider the IFS $\cF=\{f_1,f_2\}$ on $X=\{0,1\}$ with $f_1(0)=0,f_1(1)=1,f_2(0)=f_2(1)=1$.
Then $0\in\Om(0)$ since $f_1^k(0)=0$ for all $k=1,2,\dots$ but $0$ is not recurrent since $\cO_\cF(0)=\{0,1\}$ and there is no way to get close to 0 from 1.

\begin{proposition}
\label{prop: forward invariant}
The sets $\cR_\cF$ and $\Om(A)$, $A\subset X$, are forward-invariant (and so is $\Om(x)$ for all $x\in X$).
\end{proposition}

\medskip\noindent
{\bf Example: the recurrent set of contractive IFSs.}
Contractive IFSs are one of the most studied kinds of IFS in literature. 
It is well-known~\cite{Hut81,Bar88} that a contractive IFS over a complete metric space admits a unique non-empty $H$-invariant compact set $K$ and that $\Om(z)=K$ for every $z\in X$. 
Then no $x$ outside of $K$ can be recurrent and all $x\in K$ are recurrent, so for contractive IFSs we have that $\cR_\cF=K$.

\medskip\noindent
{\bf Limit sets are closed.} The following standard topological characterization of the limit set shows that every limit set is closed.
\begin{proposition}[Topological description of the limit set]
\label{prop: topological description of Om}
\[
\Om(A)
=
\bigcap_{n\ge 0}\overline{\bigcup_{|I|\ge n} \P(A)}.
\]
In particular, if $A$ is closed then so is $\Om(A)$, and so $\Om(x)$ is closed for every $x\in X$.
\end{proposition}

\medskip\noindent
{\BF $H$-invariance plays the role of invariance in IFSs with $m>1$.}
When $m=1$, invariance is a fundamental concept: when $X$ is compact, or even just when a semiflow has compact dynamics (see Section~\ref{sec: global attractor}), the global attractor is invariant and so are the chain recurrent set and each of its nodes (see Section~\ref{sec: chains}).
On the contrary, when $m>1$, having invariant sets is a strong condition and often an IFS has no invariant set.
Next sections show that, for $m>1$, the role of invariance is played by H-invariance.

\section{Global attractor of an IFS}
\label{sec: global attractor}
%
Just as in case of semiflows, the IFSs with richest asymptotics are those whose dynamics is, in some sense, ``compact''.
Of course this includes all IFSs defined over a compact space $X$ but requiring $X$ to be compact is unnecessarily strong -- it is enough that the orbits of the system are attracted by and stay within some compact region. 
For instance, this is precisely what happens in case of the celebrated Lorenz ``butterfly'' system~\cite{Lor63}. 
In this section, we give a precise definition of the IFSs that can be studied within our framework.

\medskip\noindent
{\bf Attracting and absorbing.}
    We say that a set $A\subset X$ {\bf absorbs} a set $K$ under $\cF$ if
    there exists $\tau>0$ such that $\P(K)\subset A$ for every $I\in\cI^m_\tau$.
    Given an $\eps>0$, we set
    $\text{\BF$N_\eps(A)$}=\{x : d(x,A)<\eps\}$.
    We say that $A$ {\bf attracts} $K$ under $\Phi_\cF$ if, for every $\eps>0$, there exists $\tau>0$ such that $\P(K)\subset N_\eps(A)$ for all $I\in\cI^m_\tau$.

\medskip\noindent
{\bf Global Attractor.}    
    We call {\bf global attractor} of $\cF$ a set $\cG\subset X$ that is compact and $H$-invariant and  attracts under $\cF$ each compact subset of $X$.
Next proposition shows that such a set, when it exists, is unique.
%
\begin{proposition}
    \label{prop: G is unique}
    Assume that $\cF$ has global attractors $\cG$ and $\cG'$. 
    Then $\cG=\cG'$. 
\end{proposition}
\begin{proof}
Assume that $\cG$ and $\cG'$ are two global attractors of $\cF$. By definition, both $\cG$ and $\cG'$ are compact, $H$-invariant, and attract every compact subset of $X$.
Since $\cG$ is a global attractor, it attracts $\cG'$. 
Hence, for every $\eps>0$, there is a $N>0$ such that $\P(\cG')\subset N_\eps(\cG)$ for every $I\in\cI^m_N$.
In particular, $H_\cF^n(\cG')\subset N_\eps(\cG)$ for every $n\geq N$.
Since $\cG'$ is $H$-invariant, this means that $\cG'\subset N_\eps(\cG)$ for every $\eps>0$.
Hence, $\cG'\subset\cG$.
By symmetry, $\cG=\cG'$.
%
\end{proof}

The proof of the following important characterization of the global attractor is similar to the one above and we leave it to the reader.
\begin{proposition}
    \label{prop: G_F is H-invariant}
    The global attractor is the minimal compact forward-invariant set of $\cF$ that attracts all compact subsets of $X$.
\end{proposition}

\medskip\noindent
{\bf Example: the global attractor of a contractive IFS.}
A contractive IFS over a complete metric space has a unique $H$-invariant set $K$, which coincides with its recurrent set and with the unique attractor of its Hutchinson map.
Hence, $K$ is also the global attractor of the IFS.

\medskip\noindent
{\bf Compact Dynamics.}
    We say that $\cF$ 
    has {\bf compact dynamics} if it has a global attractor.
    From now on, we denote such set by $\Gl$.

\medskip\noindent
{\bf Compact dynamics, limit sets and recurrence.} We show below that compact dynamics grants the existence of recurrent points and that these points lie in the global attractor.

\begin{lemma}
    \label{lemma: Om}
    Assume that $\cF$ has compact dynamics. Then, for every $x\in X$, $\Om(x)\cap \cR_\cF \neq \varnothing$.
    In particular, $\cR_\cF\neq\emptyset$ and, for every $x\in X$, $\Om(x)\neq\emptyset$.
\end{lemma}
\begin{proof}
    Let \(x\in X\). 
    Since \(\Gl\) is compact and \(X\) is locally compact, $\Gl$ has a compact neighbourhood \(K\). 
    Since \(\Gl\) attracts  \(\{x\}\), there exists \(n_0\ge 0\) such that $\P(x)\in K$ for $|I|\ge n_0$.
    Hence, $\Om(x)=\cap_{n\ge n_0}\overline{\bigcup_{|I|\ge n}\P(x)}$ is a decreasing intersection of non-empty compact subsets of \(K\) and so is non-empty and compact. By Proposition~\ref{prop: forward invariant} it is also forward-invariant.
    Now choose, by Zorn's lemma, a minimal non-empty compact forward-invariant subset $M\subset \Omega_F(x)$.
    We claim that every point of \(M\) is recurrent. Let \(p\in M\) and let \(z\in O_F(p)\). 
    Since \(M\) is forward-invariant, \(z\in M\), and hence $\Omega_F(z)\subset M$.
    Moreover, \(\Omega_F(z)\) is non-empty, compact and forward-invariant.
    By the minimality of \(M\), we must have $\Omega_F(z)=M$.   
    In particular, $p\in \Omega_F(z)$.
    Since this holds for every \(z\in O_F(p)\), we have \(p\in R_F\).
    Thus, \(M\subset R_F\), and, since \(M\subset\Omega_F(x)\), we obtain
    $\Omega_F(x)\cap R_F\neq\varnothing$.
    The remaining assertions follow immediately.
\end{proof}
\begin{proposition}
    \label{prop: omega(K) subset G}
    Assume that $\cF$ has compact dynamics.
    Then 
    $\Om(K)\subset\Gl$ for every compact set $K\subset X$.
    In particular, $\Om(x)\subset\Gl$ for every $x\in X$.
    Moreover, if $K$ is forward-invariant, $\Om(K)$ is $H$-invariant.
\end{proposition}
\begin{proof}


Let $K\subset X$ be compact. 
Since $\Gl$ attracts $K$, $d(\P(K),\Gl)\to 0$ for $|I|\to\infty$.
If $y\in\Om(K)$, there exist $x_n\in K$ and multiindices $I_n$ with $|I_n|\to\infty$ such that $\Pn(x_n)\to y$.
Hence, $y\in N_\eps(\Gl)$ for every $\eps>0$. 
Since $\Gl$ is closed, this is possible only if $y\in \Gl$. 
Therefore, $\Om(K)\subset \Gl$.
When $K$ is forward-invariant, a standard argument shows that every element of $\Om(K)$ is the image under some $f_i$ of an element of $\Om(K)$, so $H_\cF(\Om(K))=\Om(K)$.
\end{proof}

Notice that, in general, one cannot conclude that every point of $\Om(x)$ is recurrent; the lemma only guarantees the existence of at least one recurrent point in $\Om(x)$.


%

%

%
{\bf 
%

%
%

%
}

\begin{proposition}
\label{prop: R subset G}
Assume that $\cF$ has compact dynamics.
Then $\cR_\cF\subset\Gl$.
\end{proposition}

\begin{proof}
Let $x \in \cR_\cF$. By Proposition~\ref{prop: x in R => x in Om(x)}, $x \in \Om(x)$.
Since $\{x\}$ is compact and $\Gl$ attracts every compact subset of $X$,
Proposition~\ref{prop: omega(K) subset G} implies
$
\Om(x) \subset \Gl.
$
Hence $x \in \Gl$, and therefore
$\cR_\cF \subset \Gl.$
\end{proof}

\medskip\noindent
{\bf Trapping regions.}
Global attractors are often sets with a highly complicated structure (for instance, they are often not locally connected) and it is in general a hard problem finding out directly their existence. 
The concept of {\em trapping region} helps proving their existence.
Unlike global attractors, trapping regions are not unique and often are sets with an elementary structure, such as closed balls.
%
%
%

    We say that $Q\subset X$ is a {\bf trapping region} for $\cF$ if $Q$ is closed and forward invariant under $\Phi_\cF$.
    Given a trapping region $Q$, we denote by {\BF $\cF_Q$} the restriction of $\cF$ to $Q$, namely $\cF_Q=\{f_1|_Q,\dots,f_m|_Q\}$.
    We say that a trapping region $Q$ is {\bf global} if $Q$ absorbs, under $\Phi_\cF$, every compact set $K\subset X$.
    We denote by {\BF $\cK_\cF$} the set of all compact global trapping regions of $\cF$.
%
    

\medskip
A fundamental role of trapping regions, as illustrated by the following result, is that one can replace the whole phase space $X$ by any global trapping region when studying the asymptotics of $\Phi_\cF$.
This means that {\bf it is the topological properties of compact global trapping regions, rather than those of the whole phase space $X$, that  play a role in the qualitative dynamics of an IFS with compact dynamics}.
\begin{proposition}
    \label{prop: G is natural under restrictions}
    Assume that $\cF$ has compact dynamics and let $Q\in\cK_\cF$.
    Then $\cF_Q$ has compact dynamics and $\cG_{\cF_Q}=\cG_\cF$.
\end{proposition}
%

\begin{proof}
We first prove that
$\Gl \subset Q.$
Indeed, since $\Gl$ is compact and $Q$ absorbs every compact subset of $X$, there exists $\tau\ge 0$
such that
$\Phi_\cF^I(\Gl)\subset Q$ for every  $I\in\cI^m_\tau$.
Since $\Gl$ is $H$-invariant, for every $x\in \Gl$ and every $n\ge 1$ there exist $x_n\in \Gl$
and $I_n\in I^m$ with $|I_n|=n$ such that
$
x=\Phi_\cF^{I_n}(x_n).
$
Taking $n\ge \tau$, we get $x=\Phi_\cF^{I_n}(x_n)\in Q$. 
Hence, $\Gl\subset Q$.

As a subset of \(Q\), the set \(\Gl\) is compact, $H$-invariant under \(\cF_Q\), and attracts every compact subset of \(Q\), because it attracts every compact subset of \(X\). Hence, \(\Gl\) is a global attractor for \(\cF_Q\). 
By Proposition~\ref{prop: G is unique}, \(\cF_Q\) has compact dynamics and ${\cal G}_{\cF_Q}=\Gl$.
\end{proof}

The next result shows also another fundamental side of compact global trapping regions: they grant the existence of a global attractor.
\begin{proposition}[Existence of a global attractor]
    \label{thm: Q => G}
    Assume that $\cF$ has a compact global trapping region $Q$. 
    Then $\cF$ has compact dynamics and $\Gl=\Om(Q).$
\end{proposition}
\begin{proof}
Recall that
$\Om(Q)=\bigcap_{n\ge 0} E_n$ and 
$E_n=\bigcup_{|I|\ge n}\P(Q).$
Since $Q$ is forward-invariant, for every $n\ge 0$ we have
$
E_n=\bigcup_{|I|=n}\P(Q).
$
Hence, each $E_n$ is a compact subset of $Q$, and
$
E_{n+1}\subset E_n$ for all $n\ge 0$.
Therefore $\Om(Q)=\bigcap_{n\ge 0}E_n$ is a non-empty compact subset of $Q$.

By Proposition~\ref{prop: forward invariant}, $\Omega_\cF(Q)$ is forward-invariant.
We now show that $\Om(Q)$ is also $H$-invariant. 
By its forward-invariance, $H_\cF(\Om(Q))\subset \Om(Q).$
It remains to prove the opposite inclusion.
Fix $x\in \Om(Q)$. 
Since $x\in E_n$ for every $n\ge 1$, for each $n\ge 1$ there exist $q_n\in Q$ and a multiindex $I_n=i_nJ_n, |I_n|=n$, such that $x=\Phi_\cF^{I_n}(q_n)=f_{i_n}\bigl(\Phi_\cF^{J_n}(q_n)\bigr).$
Set $y_n=\Phi_\cF^{J_n}(q_n).$
Then $y_n\in E_{n-1}\subset Q$ and $x=f_{i_n}(y_n).$
Since $Q$ is compact and $\{1,\dots,m\}$ is finite, passing to a subsequence if necessary we may assume that $y_n\to y\in Q$ and $i_n=i$ for all $n$.
We claim that $y\in \Om(Q)$. 
Indeed, fix $p\ge 0$. 
If $n$ is sufficiently large, then $n-1\ge p$, and so $y_n\in E_{n-1}\subset E_p$.
Since $E_p$ is closed, passing to the limit yields $y\in E_p$. 
Since $p$ was arbitrary, we conclude that $y\in \bigcap_{p\ge 0}E_p=\Om(Q)$.
By continuity of $f_i$, $x=\lim_{n\to\infty} f_i(y_n)=f_i(y)$ and, since $y\in \Om(Q)$, this proves that $x\in H_\cF(\Om(Q))$. 
Hence, $\Om(Q)\subset H_\cF(\Om(Q))$.
Therefore $H_\cF(\Om(Q))=\Om(Q)$, namely $\Om(Q)$ is $H$-invariant.

Now we prove that $\Om(Q)$ attracts every compact subset of $X$. 
Let $K\subset X$ be compact and let $\varepsilon>0$. 
Since $Q$ is a global trapping region, there exists $N_0\ge 0$ such that $\P(K)\subset Q$ for every $I\in\cI^m_{N_0}$.
Since the family $(E_n)_{n\ge 0}$ is decreasing and $\Om(Q)=\bigcap_{n\ge 0}E_n,$ compactness implies that there exists $n\ge 0$ such that $d(E_n,\Om(Q))<\varepsilon.$
Now, let $I\in\cI^m$ with $|I|\ge N_0+n$ and write $I=JL$ with $|L|=N_0$ and $|J|\ge n$. 
Then
$\P(K)=\Phi_\cF^J(\Phi_\cF^L(K))\subset \Phi_\cF^J(Q)\subset E_n.
$
Hence
\[
d(\P(K),\Om(Q))\le d(E_n,\Om(Q))<\varepsilon
\]
for every $|I|\ge N_0+n$. This proves that $\Om(Q)$ attracts every compact subset of $X$.
Hence, $\cF$ has compact dynamics and $\Gl=\Om(Q)$.
\end{proof}
\begin{lemma}
    \label{lemma: small trapping regions}
    Assume that \(\cF\) has compact dynamics. 
    Then, for every \(\varepsilon>0\), there exists a compact global trapping region $Q\subset N_\varepsilon(\Gl)$.
\end{lemma}
\begin{proof}
    Choose a compact neighbourhood \(K\) of \(\Gl\) such that $\Gl\subset \operatorname{int}K\subset K\subset N_\varepsilon(\Gl)$.
    Since \(\Gl\) attracts \(K\), there exists \(N\ge1\) such that
    $\P(K)\subset \operatorname{int}K$ for $|I|\ge N$.
    Set $Q=\overline{\cup_{|I|\ge N}\P(K)}$.
    Then \(Q\subset K\), hence \(Q\) is compact and \(Q\subset N_\varepsilon(\Gl)\).

    Now, for each generator \(f_i\),
\[
f_i\!\left(\bigcup_{|I|\ge N}\Phi_F^I(K)\right)
\subset
\bigcup_{|J|\ge N+1}\Phi_F^J(K)
\subset
\bigcup_{|J|\ge N}\Phi_F^J(K).
\]
    By continuity, $f_i(Q)\subset Q$.
    Thus, \(Q\) is a trapping region.
    Finally, let \(C\subset X\) be compact. 
    Since \(\Gl\) attracts \(C\) and \(K\) is a neighbourhood of \(\Gl\), there exists \(M\ge0\) such that $\P(C)\subset K$ for $|I|\ge M$.
    If \(|L|\ge M+N\), split \(L\) into two pieces, first of length at least \(M\) and then of length at least \(N\). 
    The first part sends \(C\) into \(K\), and the remaining part sends it into \(Q\). 
    Hence, $\Phi_\cF^L(C)\subset Q$ for $|L|\ge M+N$.
    Therefore, \(Q\) absorbs every compact subset of \(X\), so it is a compact global trapping region.
\end{proof}
%



%
\begin{theorem}
    \label{thm: compact dynamics = compact gl tr reg}
    $\cF$ has compact dynamics if and only if it has a compact global trapping region $Q$.
\end{theorem}
\begin{proof}
If $\cF$ has a compact global trapping region, then $\cF$ has compact dynamics by Proposition~\ref{thm: Q => G}.
Conversely, if $\cF$ has compact dynamics, then, by Lemma~\ref{lemma: small trapping regions}, for every
$\varepsilon>0$ there exists a compact global trapping region
$Q\subset N_\varepsilon(\Gl)$. Hence $\cF$ has a compact global trapping
region.
\end{proof}

\section{The chain graph}
\label{sec: chains}
The concept of $\eps$-chain was introduced in the theory of dynamical systems by Charles Conley in 1972~\cite{Con72b} and imported to the theory of IFSs by Barnsley and Vince~\cite{BV13} as follows.

\medskip\noindent
{\BF Chains for IFSs.}
Let $\varepsilon>0$.
An {\BF$\varepsilon$-chain} from $x$ to $y$ is a finite sequence of points of $X$ $x=x_0,x_1,\dots,x_k=y$
together with a sequence of indices $i_0,\dots,i_{k-1}$ such that
$
d(f_{i_n}(x_n),x_{n+1})<\varepsilon$ for $n=0,\dots,k-1$.



\medskip\noindent
{\bf Being downstream.}
Given $x,y\in X$, we say that $y$ is {\bf downstream} from $x$ if, for every $\eps>0$, there is an $\eps$-chain from $x$ to $y$.
In this case, we write $x\Cto y$ and we set 
$$
\text{\BF$\cC$}=\{(x,y):x\Cto y\}\subset X\times X
$$
for the corresponding binary relation.
\begin{proposition}
The relation $\cC$ is closed and transitive and contains all $(x,y)$ with $y\in\cO_\cF(x)$.
\end{proposition}
\begin{proof}
We prove here only that $\cC$ is closed and leave to the reader the simple proof of the other two points.
Let
\((x_n,y_n)\in \cC\) and \((x_n,y_n)\to(x,y)\) and fix an
\(\varepsilon>0\). 
Since the family \(\cF\) is finite and every \(f_i\) is continuous, for all sufficiently large \(n\) we have that $d(f_i(x),f_i(x_n))<\varepsilon/3$ for $i=1,\ldots,m$,
and also \(d(y_n,y)<\varepsilon/3\). 
Choose such an \(n\). 
Since \((x_n,y_n)\in \cC\), there is an \(\varepsilon/3\)-chain
$x_n=z_0,z_1,\ldots,z_k=y_n$
from \(x_n\) to \(y_n\). 
If \(k=0\) for infinitely many $n$, then \(x_n=y_n\), so that \(x=y\), and the trivial chain proves \(x\Cto y\). 
Otherwise, for $n$ large enough we have \(k\ge1\) and, replacing the initial point \(x_n\) by \(x\) and the endpoint \(y_n\) by \(y\) gives
$x,z_1,\ldots,z_{k-1},y.$
The first jump has size \(<2\varepsilon/3\), the last jump has size \(<2\varepsilon/3\), and all intermediate jumps still have size \(<\varepsilon/3\). 
Thus, this is an \(\varepsilon\)-chain from \(x\) to \(y\). 
Since \(\varepsilon>0\) was arbitrary, \(x\Cto y\).
Therefore, \(\cC\) is closed.
\end{proof}

\medskip\noindent
{\bf Chain-recurrence, chain nodes, chain edges and chain graph.}
Two points $x,y\in X$ are {\BF chain-equivalent} if $x\Cto y$ and $y\Cto x$.
In this case, we write $x\Ceq y$.
The set of {\BF chain-recurrent} points is then defined as
$$
\text{\BF$\cR_\cC$} = \{x\in X\;:\;x\Ceq z\text{ for all }z\in\cO_\cF(x)\}.
$$
%
Chain-equivalence, restricted to the chain-recurrent set, is an equivalence relation.
We call {\bf chain node} (or, simply, a {\bf node}) each of the equivalence classes of this equivalence relation.
Equivalently, each chain node $N$ is a maximal subset of $\cR_\cC$ such that each of the points of $N$ is chain-equivalent to every other point of $N$.

Given two distinct chain nodes $M,N$, we say that there is a {\bf chain edge} (or, simply, an {\bf edge}) from $M$ to $N$ if there is an $x_0\in M$ and a $y_0\in N$ such that $x_0\Cto y_0$.
By transitivity, this means that $x\Cto y$ for every $x\in M$ and $y\in N$.
In this case, we write $M\Cto N$.

We call {\bf chain graph} the directed graph $\Gamma_\cC$ having the chain nodes as its nodes and the chain edges as its edges.
Notice that, since $\cC$ is transitive, $\Gamma_\cC$ has no directed cycles and that $\Gamma_\cC$ can have continuously many nodes.

\medskip
Notice that this definition of chain-recurrence is stronger than the one given by Barnsley and Vince in~\cite{BV13}.
A point is ``BV chain-recurrent'' if, for every $\eps>0$, there is an $\eps$-chain for $\cF$ from $x$ to itself.
As the example below shows, the BV chain nodes lack forward-invariance and so are not suitable to represent elementary blocks with some kind of recurrent dynamics, which is our main goal. 

\medskip
Let $X=[0,1]\cup[2,3]\subset\mathbb R$ with the Euclidean metric, and consider the IFS 
$\cF=\{f,g\}$, $f=\operatorname{id}_X$, $g(x)=2$.
Every point of \(X\) is BV chain-recurrent: indeed, for every \(x\in X\) and every \(\varepsilon>0\), the one-step chain $x,\ x$ using the map \(f\) is an exact \(\varepsilon\)-chain from \(x\) to itself.

If \(x\in[0,1]\), then $2=g(x)\in \cO_\cF(x)$.
However, for \(\varepsilon<1\), there is no \(\varepsilon\)-chain from \(2\) back to \(x\). Indeed, starting in the component \([2,3]\), both generators keep the next point within distance \(<\varepsilon\) of \([2,3]\), and such a chain can never jump across the gap to
\([0,1]\).
Hence, points in \([0,1]\) return to themselves only because one is
allowed to ignore the generator \(g\), which sends them irreversibly to the other component. 

Ultimately, there are two BV nodes: the interval $[0,1]$ and the interval $[2,3]$. 
Moreover, there is an edge from the first to the latter, so that the BV graph is $[0,1]\to[2,3]$.
On the other side, $g([0,1])\subset[2,3]$, namely the node $[0,1]$ is not forward-invariant (and so not $H$-invariant either).
Points in $[0,1]$ should therefore not be regarded as recurrent
for the full IFS dynamics.
On the contrary, with our definition only points in $[2,3]$ are chain-recurrent and $[2,3]$ is a node, so the graph consists of a single node.
Since $f([2,3])=[2,3]$ and $g([2,3])=\{2\}$, this node is forward-invariant (in fact, it is moreover $H$-invariant).

\medskip
Finally, notice that both our chain-recurrent set and the BV chain recurrent set coincide with the standard chain-recurrent set for $m=1$.

\medskip\noindent
{\bf Chain-recurrence and recurrence.} Next proposition shows that, just as in the $m=1$ case, every recurrent point is chain-recurrent.
\begin{proposition}
    \label{prop: at least one node}
    For any IFS $\cF$, $\cR_\cF\subset\cR_\cC$.
    In particular, $\cR_\cC\neq\varnothing$ if $\cF$ has compact dynamics.
\end{proposition}
%
\begin{proof}
    Let \(x\in \cR_\cF\) and \(z\in \cO_\cF(x)\). Since \(z\in \cO_\cF(x)\), the exact orbit segment from \(x\) to \(z\) gives $x\succeq_{C_F} z$.
    On the other hand, since \(x\in \cR_\cF\), we have $x\in\Omega_F(z)$.
    Therefore, for every \(\varepsilon>0\), there is an orbit point $\P(z)\) with $d(\P(z),x)<\varepsilon$.
    Following the exact orbit segment from \(z\) to \(\P(z)\), and then making the final \(\varepsilon\)-jump to \(x\), gives an \(\varepsilon\)-chain from \(z\) to \(x\). 
    Hence $z\Cto x$, so that \(x\Ceq z\) for every \(z\in \cO_\cF(x)\) and $x\in \cR_{\cC}$.
    If $\cF$ has compact dynamics, $\cR_\cF$ is non-empty by Lemma~\ref{lemma: Om} and so $\cR_\cC$ is non-empty as well. 
\end{proof}
Unlike $\cR_\cF$, though, $\cR_\cC$ is always closed: 
\begin{proposition}
    For any IFS $\cF$, $\cR_\cC$ and each of its nodes are closed.
\end{proposition}
\begin{proof}
    First, we show that \(\cR_\cC\) is closed. 
    Let \(x_n\in \cR_\cC\) and \(x_n\to x\).
    If \(z\in \cO_\cF(x)\), write \(z=\Phi_F^I(x)\). 
    Set $z_n=\P(x_n)$.
    Then \(z_n\to z\). 
    Since \(x_n\in \cR_\cC\), we have $x_n\Ceq z_n$.
    Given that \(\cC\) is closed, passing to the limit gives $x\Ceq z$.
    Since \(z\in \cO_\cF(x)\) was arbitrary, \(x\in \cR_\cC\), so
    \(\cR_\cC\) is closed.
    Now, let \(N\) be the node containing \(x_0\in \cR_\cC\). 
    Then $N=\cR_\cC\cap\{y:x_0\Cto y\}\cap\{y:y\Cto x_0\}$.
    Each set on the right is closed, so \(N\) is closed.
%
\end{proof}

\medskip\noindent
{\bf Chain-recurrence and limit sets.}
When $m=1$ and $\cF$ has compact dynamics, the limit set of each point is a (non-empty) chain-equivalent set, namely it is entirely contained in a single chain node.
For $m>1$, $\Om(x)$ is not necessarily a subset of $\cR_{\cC}$ and $\Om(x)\cap\cR_{\cC}$ is not necessarily chain-equivalent.
The following is our most general result in this regard.

\begin{proposition}
    \label{prop: Om(x) has some rec pt}
    Assume that $\cF$ has compact dynamics.
    Then, for each $x\in X$, $\Om(x)$ has at least one non-empty minimal compact forward-invariant subset and each such minimal set is chain-equivalent, namely it is contained in some chain node.
\end{proposition}
\begin{proof}
    The existence of a non-empty compact forward-invariant $M\subset\Om(x)$ is shown in the proof of Lemma~\ref{lemma: Om}.
    Because of the forward invariance, for any $y\in M$ we have that $\Om(y)\subset M$.
    Since $\Om(y)$ is forward-invariant and $M$ is minimal, $\Om(y)=M$.
    Hence, for any distinct points $y,z\in M$, we have that $y\in\Om(z)$ and $z\in\Om(y)$.
    Hence, $y\Cto z$ and $z\Cto y$, namely $y\Ceq z.$
\end{proof}

\medskip\noindent
{\bf Chain nodes and the global attractor.}
We show below that all chain nodes of an IFS $\cF$ with compact dynamics lie inside its global attractor $\Gl$.
Hence, the qualitative dynamics of $\cF$ can be entirely read from the restriction of $\cF$ to $\Gl$, which is a compact set. 
In other words, an IFS with compact dynamics has the same qualitative dynamics of an IFS on a compact set.

\begin{lemma}[Localization of chains near the attractor]
\label{lemma: localization}
Let \(\cF\) have compact dynamics.
Then, for every $\eps>0$, there exists \(\delta>0\) such that every \(\delta\)-chain starting in \(\Gl\) lies in \(N_\eps(\Gl)\).
\end{lemma}
\begin{proof}
Choose a compact neighbourhood \(K\) of \(\Gl\) such that $\Gl\subset \operatorname{int}K\subset K\subset N_\varepsilon(\Gl)$.
Since \(\Gl\) attracts \(K\), there exists \(N\ge1\) such that $\P(K)\subset \operatorname{int}K$ for $|I|=N$.
We claim that, for \(\delta>0\) sufficiently small, every \(\delta\)-chain
starting in \(\Gl\) remains in \(K\). 
Indeed, by compactness, finiteness of the set of words of length at most \(N\) and uniform continuity on \(K\), we may choose \(\delta>0\) so small that every \(\delta\)-chain of length at most \(N\) starting in \(\Gl\) remains in \(K\) and every \(\delta\)-chain of length \(N\) whose previous points lie in \(K\) has endpoint in \(K\).

Suppose now that a \(\delta\)-chain $x_0,x_1,\ldots$ starts at \(x_0\in \Gl\) and leaves \(K\). 
Let \(k\) be the first exit time. 
If \(k\le N\), this contradicts the first property above. 
If \(k>N\), then the last \(N\) steps before the exit start in \(K\) and all intermediate points before \(x_k\) lie in \(K\); by the second property, \(x_k\in K\), again a contradiction. Hence every such chain
remains in \(K\subset N_\varepsilon(\Gl)\).
\end{proof}

\begin{lemma}[Restriction to the global attractor]
\label{lemma: restriction}
Assume that $\cF$ has compact dynamics. 
Then, for any $x,y\in\Gl$, $y$ is downstream from $x$ under $\cF$ if and only if $y$ is downstream from $x$ under $\cF|_{\Gl}$.
\end{lemma}
\begin{proof}
We prove the non-trivial inclusion.
Assume that $y$ is downstream $x$ under $\cF$.
Fix an \(\varepsilon>0\) and choose a compact neighbourhood \(K\) of \(\Gl\).
By the uniform continuity of the finitely many maps \(f_i|_K\), there is an \(\eta>0\) so small that
\[
u,v\in K,\quad d(u,v)<\eta
        \quad\Longrightarrow\quad
d(f_i(u),f_i(v))<\varepsilon/3
\]
for every \(i=1,\ldots,m\).
We can assume without loss of generality that \(\eta<\varepsilon/3\)
and that $N_\eta(\Gl)\subset K$.
By Lemma~\ref{lemma: localization}, there exists \(\delta>0\) such that every \(\delta\)-chain starting in \(\Gl\) remains in \(N_\eta(\Gl)\).
Set \(\rho=\min\{\delta,\eta,\varepsilon/3\}\). 
Since \(x\Cto y\), there is a \(\rho\)-chain $x=x_0,x_1,\ldots,x_n=y$
from \(x\) to \(y\). 
By the choice of \(\delta\), every \(x_j\) lies in
\(N_\eta(\Gl)\). Choose \(z_j\in \Gl\) with $d(x_j,z_j)<\eta$
and take \(z_0=x\), \(z_n=y\).
If \(i_j\) is the index used in the \(j\)-th step of the original chain,
then
\[
\begin{aligned}
d(f_{i_j}(z_j),z_{j+1})
&\le d(f_{i_j}(z_j),f_{i_j}(x_j))
   + d(f_{i_j}(x_j),x_{j+1})
   + d(x_{j+1},z_{j+1})  \\
&< \varepsilon/3+\rho+\eta
< \varepsilon .
\end{aligned}
\]
Thus, $z_0,z_1,\ldots,z_n$
is an \(\varepsilon\)-chain inside \(\Gl\) from \(x\) to \(y\). 
Since \(\varepsilon>0\) was arbitrary, $y$ is downstream from $x$ under $\cF|_{\Gl}$.
\end{proof}
\begin{theorem}
    \label{prop: C = C glob}
    Assume that $\cF$ has compact dynamics. 
    Then the chain-recurrent set, the chain nodes and the chain graph of \(\cF\) coincide with those of \(\cF|_{\Gl}\).
\end{theorem}
\begin{proof}
We show first that  \(\cR_{\cC}\subset \Gl\). 
Indeed, if \(x\in \cR_{\cC}\), then, by Lemma~\ref{lemma: Om}, there exists \(p\in \Om(x)\cap \cR_\cF\).
By 
Proposition~\ref{prop: R subset G}, \(p\in \Gl\). 
Since \(x\in \cR_\cC\), every point of \(\cO_\cF(x)\) is chain-equivalent to \(x\). 
In particular, if \(p_n=\Pn(x)\) and \(|I_n|\to\infty\), then \(p_n\Cto x\) for all \(n\). Passing to a convergent subsequence \(p_n\to p\in\Om(x)\) and using the closedness of \(\cC\), we obtain \(p\Cto x\).
By the localization lemma applied to arbitrarily small neighbourhoods of \(\Gl\), every point downstream from a
point of \(\Gl\) lies in \(\Gl\). 
Hence, \(x\in \Gl\) and so $\cR_\cC\subset\Gl$.

Lemma~\ref{lemma: restriction} shows that the two chain relations agree on
\(\Gl\times \Gl\). Since all chain-recurrent points of $\cF$ lie in
\(\Gl\), the recurrent sets coincide. Consequently their equivalence
classes, hence their nodes, coincide. The edge relation also coincides
because it is defined entirely in terms of the same relation restricted
to \(\Gl\times \Gl\).
\end{proof}

%
\begin{corollary}
    Assume that $\cF$ has compact dynamics. 
    Then its chain-recurrent set and all of its nodes are compact.
\end{corollary}
%
    


\medskip\noindent
{\BF Chain-recurrence, forward-invariance and $H$-invariance.}
When $m=1$, $H$-invariance coincides with invariance and the chain-recurrent set of a semiflow with compact dynamics and all of its nodes are invariant~\cite{DLY25,DLY26a}.
The reader can verify that forward-invariance still holds for the chain-recurrent set and each node when $m>1$.
These sets are not necessarily $H$-invariant but we have the following weaker property.
\begin{theorem}
    \label{thm: N is H-invariant}
    Assume that $\cF$ has compact dynamics and let $N$ be one of its chain nodes.
    Then $N$ is forward-invariant and $\Om(N)$ is a non-empty compact $H$-invariant chain-equivalent subset of $N$.
\end{theorem}
\begin{proof}
    We know that $N$ is compact and forward-invariant.
    Hence, by Lemma~\ref{lemma: Om} and Proposition~\ref{prop: omega(K) subset G}, $\Om(N)$ is $H$-invariant and non-empty.
    Since $N$ is chain-equivalent and $\cC$ is closed, then $\Om(N)$ is chain-equivalent as well.
\end{proof}

\medskip\noindent
{\bf Recurrent and gradient dynamics of an IFS.}
The following result shows that chain nodes can be considered as sets where the dynamics of $\cF$ is of recurrent type.
\begin{proposition}
    Assume that \(\cF\) has compact dynamics and let $N$ be a chain node of $\cF$. 
    Then $N\cap \cR_\cF\neq\varnothing$.
\end{proposition}
\begin{proof}
    By Theorem~\ref{thm: N is H-invariant}, the set \(\Om(N)\) is non-empty, compact, \(H\)-invariant, and contained in \(N\). 
    Choose \(x\in \Om(N)\).
    By Lemma~\ref{lemma: Om}, there exists $p\in \Om(x)\cap \cR_\cF$.
    Since \(\Om(N)\) is forward-invariant and \(x\in\Om(N)\), we have $\Om(x)\subset \Om(N)$.
    Therefore $p\in \Om(N)\subset N$.
    Hence \(p\in N\cap \cR_\cF\).
\end{proof}
In other words, $\cC$ allows to sort canonically $\cR_\cF$ into elementary blocks. 
The points that are chain-recurrent but not recurrent make $\cR_{\cC}$ and its nodes closed, unlike $\cR_\cF$, which in general is not.

Now, notice that, for each $x\in X$, each point of $\Om(x)$ is downstream from $x$ and $\Om(x)$, by Proposition~\ref{prop: Om(x) has some rec pt}, intersects at least one chain node.
Hence, each $x$ is upstream from some chain node. 
This shows that every $x\not\in\cR_\cC$ has a dynamics of gradient type, in the sense that there are points $y\in X$ asymptotically reachable by $x$ such that $y$ cannot chain back to $x$ -- if every such $y$ would chain back to $x$ then it easily follows that $x\in\cR_\cC$, against the initial assumption.
Hence, $\cC$ allows naturally to separate the dynamics of $\cF$ into recurrent and gradient components.

\section{The Hutchinson map}
\label{sec: Hutch}


\medskip\noindent{\bf The space of all compact subspaces.}
We denote by {\BF$\KX$} the space of all non-empty compact subspaces of $X$.
We endow this space with the 
 {\bf Hausdorff metric}
\[
\text{\BF$\dH(A,B)$}=\max\Bigl\{\sup_{a\in A} d(a,B),\ \sup_{b\in B} d(b,A)\Bigr\},\quad A,B\in\KX.
\]
When $(X,d)$ is compact or locally compact, so is $(\KX,d_H)$.
This space is often referred to in literature as the {\bf hyperspace} of $X$.

\medskip\noindent
{\bf The Hutchinson map.}
By {\bf Hutchinson map} of $\cF$ we mean the map $H_\cF:\KX\to\KX$ defined by
\[
\text{\BF$H_\cF(K)$}=\bigcup_{i=1}^m f_i(K),\quad K\in\KX.
\]


\medskip\noindent
{\BF Trapping regions and global attractor of $H_\cF$.}
The map $H_\cF$ defines a discrete-time semiflow, i.e. an IFS with $m=1$, on $\KX$.
Hence, all results of the previous sections apply to it.

\begin{proposition}[Trapping regions and the Hutchinson map]
\label{prop: Q <-> K(Q)}
The following hold:
\begin{enumerate}
\item If \(Q\subset X\) is a compact global trapping region for \(\cF\), then \(\mathcal K(Q)\) is a compact global trapping region for \(H_\cF\).

\item Conversely, if \(\mathcal Q\subset\mathcal K(X)\) is a compact
global trapping region for \(H_\cF\), then
$Q=\cup_{A\in\mathcal Q} A$
is a compact global trapping region for \(\cF\).
\end{enumerate}
\end{proposition}

\begin{proof}
We use the identity $H_\cF^n(A)=\bigcup_{|I|=n}\P(A)$.

Assume first that \(Q\subset X\) is a compact global trapping region for
\(\cF\). 
Then \(\mathcal K(Q)\) is compact. 
Since \(Q\) is forward-invariant,
\(H_\cF(A)\subset Q\) for every \(A\in\mathcal K(Q)\), so \(\mathcal K(Q)\) is forward-invariant. 
Let \(\mathcal C\subset\mathcal K(X)\) be compact.
Then $U=\bigcup_{A\in\mathcal C}A$
is compact in \(X\). 
Since \(Q\) absorbs \(U\), there exists \(N\) such that \(\Phi_\cF^I(U)\subset Q\) for all \(|I|\ge N\). 
Therefore,
\(H_\cF^n(A)\subset Q\) for every \(A\in\mathcal C\) and every \(n\ge N\).
Thus, \(H_\cF^n(\mathcal C)\subset\mathcal K(Q)\) for all \(n\ge N\), so \(\mathcal K(Q)\) is a compact global trapping region for \(H_\cF\).

Conversely, assume that \(\mathcal Q\subset\mathcal K(X)\) is a compact
global trapping region for \(H_\cF\). 
The union $Q=\bigcup_{A\in\mathcal Q}A$ is compact. 
If \(x\in Q\), then \(x\in A\) for some \(A\in\mathcal Q\).
Since \(\mathcal Q\) is forward-invariant under \(H_\cF\), then $f_i(x)\in H_\cF(A)\subset Q$
for every \(i\), so \(Q\) is forward-invariant. 
Finally, let \(B\subset X\) be compact. 
Since \(\{B\}\) is a compact subset of \(\mathcal K(X)\), there exists \(N\) such that $H_\cF^n(B)\in\mathcal Q$ for $n\ge N$.
Hence, for every \(|I|=n\ge N\), $\P(B)\subset H_\cF^n(B)\subset Q$.
Thus \(Q\) absorbs every compact subset of \(X\), and is therefore a
compact global trapping region for \(\cF\).
\end{proof}

\begin{theorem}
    \label{thm: F comp dyn iff H comp dyn}
    \(\cF\) has compact dynamics if and only if \(H_\cF\) has compact dynamics.
    Moreover, when these conditions hold, if \(\mathcal G_H\subset \mathcal K(X)\) is the global attractor of \(H_\cF\), then $\mathcal G_H\subset \mathcal K(\Gl)$ and $\Gl=\bigcup_{A\in \mathcal G_H} A$.
\end{theorem}

\begin{proof}
Assume first that \(\cF\) has compact dynamics. 
By Theorem~\ref{thm: compact dynamics = compact gl tr reg}, \(\cF\) has a compact global trapping region \(Q\subset X\). 
By Proposition~\ref{prop: Q <-> K(Q)}, \(\mathcal K(Q)\) is a compact global trapping region
for \(H_\cF\). 
Hence, \(H_\cF\) has compact dynamics.
Let \(\mathcal G_H\) be the global attractor of \(H_\cF\). 
We first prove that $\mathcal G_H\subset \mathcal K(\Gl)$.

Let \(A\in \mathcal G_H\). Since \(\mathcal G_H\) is invariant under \(H_\cF\), for every \(n\ge1\) there exists \(A_n\in \mathcal G_H\) such that $A=H_\cF^n(A_n)$.
Since \(\mathcal G_H\) is compact in \(\mathcal K(X)\), the set $U=\bigcup_{B\in\mathcal G_H}B$ is compact in \(X\). 
Since \(\Gl\) attracts every compact subset of \(X\), for every \(\varepsilon>0\) there exists \(N\ge1\) such that $\P(U)\subset N_\varepsilon(\Gl)$ for all $|I|\ge N$.
Using the identity $H_\cF^n(B)=\bigcup_{|I|=n}\P(B)$,
we get $H_\cF^n(B)\subset N_\varepsilon(\Gl)$ for every $B\in\mathcal G_H,\ n\ge N$.
In particular, for \(n\ge N\), $A=H_\cF^n(A_n)\subset N_\varepsilon(\Gl)$.
Since \(\varepsilon>0\) was arbitrary and \(\Gl\) is closed, it follows that \(A\subset \Gl\). Thus \(A\in\mathcal K(\Gl)\), and therefore $\mathcal G_H\subset \mathcal K(\Gl)$.

Now, set $G=\bigcup_{A\in\mathcal G_H}A$.
Again, \(G\) is compact. 
We show that \(G\) is \(H\)-invariant under \(\cF\). 
Since \(\mathcal G_H\) is invariant under \(H_\cF\),
\[
H_\cF(G)
=
\bigcup_{A\in\mathcal G_H} H_\cF(A) =
\bigcup_{B\in H_\cF(\mathcal G_H)} B 
=
\bigcup_{B\in \mathcal G_H} B 
=G.
\]
Hence, \(G\) is \(H\)-invariant.
Now, let
\(K\subset X\) be compact. 
Since \(\mathcal G_H\) attracts the singleton \(\{K\}\subset\mathcal K(X)\), for every \(\varepsilon>0\) there exists \(N\ge1\) such that $
d_H\bigl(H_\cF^n(K),\mathcal G_H\bigr)<\varepsilon$ for $n\ge N$.
Therefore, for every \(n\ge N\), there exists \(A_n\in\mathcal G_H\)
such that $H_\cF^n(K)\subset N_\varepsilon(A_n)\subset N_\varepsilon(G)$.
Equivalently, $\P(K)\subset N_\varepsilon(G)$ for all $|I|=n\ge N$.
Thus, \(G\) attracts every compact subset of \(X\). 
Since \(G\) is compact, \(H\)-invariant and attracts every compact subset of \(X\), $G$ is the global attractor of $\cF$. Hence, $\Gl=\bigcup_{A\in\mathcal G_H}A$.

Conversely, assume that \(H_\cF\) has compact dynamics. 
By Theorem~\ref{thm: compact dynamics = compact gl tr reg} applied to the discrete semiflow \(H_\cF\), there exists a compact global trapping region $\mathcal Q\subset\mathcal K(X)$ for \(H_\cF\). 
By Proposition~\ref{prop: Q <-> K(Q)}, $Q=\bigcup_{A\in\mathcal Q}A$
is a compact global trapping region for \(\cF\). 
Hence, again by Theorem~\ref{thm: compact dynamics = compact gl tr reg}, \(\cF\) has compact dynamics.
\end{proof}

\medskip\noindent
{\BF Chains of $H_\cF$.}
Let $\eps>0$. 
An $\eps$-chain from $K$ to $L$ is a finite sequence
$K=K_0,K_1,\dots,K_n=L$
such that
$\dH\bigl(H_\cF(K_j),K_{j+1}\bigr)<\eps
\quad\text{for all }j=0,\dots,n-1$.
We write that $K\Hto L$ if, for every $\eps>0$, there is an $\eps$-chain from $K$ to $L$.
In this case, we say that $L$ is $H_\cF$-downstream from $K$.
We set
\[
{\cal C}_{H_\cF} = \{(K,L)\in \KX\times \KX:\ K\Hto L\}.
\]
This relation is closed and transitive for any semiflow~\cite{DLY25,DLY26a}.
Chain nodes, edges and graph of $H_\cF$ are defined in the same way as for $\cF$.


\medskip\noindent
{\bf Contractive IFSs.}
We recall the following seminal result of Hutchinson.
\begin{thmX}[Hutchinson, 1981~\cite{Hut81}]
    \label{thm: Hutch}
    Assume that $\cF$ is contractive and $X$ is a complete metric space.
    Then $H_\cF$ is a contraction in $\KX$ and so, in particular, it has a unique fixed point $K$ and every point in $\KX$ asymptotes to $K$ under $H_\cF$. 
\end{thmX}
In our framework, Hutchinson's result can be expressed as follows.
%
%
%
\begin{proposition}
\label{prop: R_C=K}
Assume that $\cF$ is a contractive IFS with compact dynamics and let $K\in\KX$ be the unique fixed point of $H_\cF$.
Then $\Gl=K\subset X$ and $K$ is also the only chain node of $\cF$, namely $\cR_\cC=K$ and all points of $K$ are chain-equivalent to each other.
\end{proposition}
The proof is essentially the same as the proof of Theorem~\ref{thm: LY}.

\medskip\noindent
{\BF Relations between $\cC$ and ${\cal C}_{H_\cF}$.}
Recall that $\Gl$ is $H$-invariant under $\Phi_\cF$.
Hence, when $\cF$ has compact dynamics, the map $H_\cF$ has at least one fixed point. 
Moreover, by Theorem~\ref{thm: N is H-invariant}, to each node $N$ of $\cF$ it corresponds a fixed point of 
$H_\cF$ given by the compact set $\Om(N)$.
Hence, we have a map $N\mapsto\Node_{{\cal C}_{H_\cF}}(\Om(N))$  from nodes of $\cC$ to nodes of ${\cal C}_{H_\cF}$.
We show in the next proposition that this map is injective.
In particular, this shows that the chain graph of $H_\cF$ has at least as many nodes as the chain graph of $\Phi_\cF$.

\begin{lemma}
    \label{prop: N->Node(N) is injective}
    Assume that $\cF$ has compact dynamics and let \(N_1,N_2\) be two distinct chain nodes of $\cF$. 
    Then
    $
    \operatorname{Node}_{{\cal C}_{H_F}}(\Om(N_1))
    \neq
    \operatorname{Node}_{{\cal C}_{H_F}}(\Om(N_2)).
    $
\end{lemma}
\begin{proof}
Let \(M,N\) be chain nodes of
$\cF$, regarded as compact subsets of \(X\). 
If $M\Hto N$, then $M\Cto N$.
Indeed, fix \(x\in M\), \(y\in N\), and \(\varepsilon>0\). 
Since \(M\Hto N\), there is an \(\varepsilon\)-chain in \(\mathcal K(X)\), 
\[
M=K_0,K_1,\ldots,K_r=N,
\]
from \(M\) to \(N\). 
If \(r=0\), then \(M=N\), and there is nothing to prove. 
Otherwise, starting from \(y_r=y\in K_r\) and using $d_H(H_F(K_{j-1}),K_j)<\varepsilon$, we choose inductively points \(y_{j-1}\in K_{j-1}\) and indices
\(i_{j-1}\in\{1,\ldots,m\}\) such that $d(f_{i_{j-1}}(y_{j-1}),y_j)<\varepsilon$ for $j=1,\ldots,r$.
Thus, $y_0,y_1,\ldots,y_r=y$ is an \(\varepsilon\)-chain for $\cF$ from some point \(y_0\in M\) to
\(y\). 
Since \(x,y_0\in M\) and \(M\) is a chain node, we also have \(x\Cto y_0\). 
Hence, there is an \(\varepsilon\)-chain from \(x\) to \(y_0\). Concatenating the two chains gives an \(\varepsilon\)-chain from \(x\) to \(y\). 
Since \(\varepsilon>0\) was
arbitrary, \(x\Cto y\). 
Therefore \(M\Cto N\).

Now assume, by contradiction, that
$\operatorname{Node}_{{\cal C}_{H_F}}(\Om(N_1))
=
\operatorname{Node}_{{\cal C}_{H_F}}(\Om(N_2)).
$
Then $\Om(N_1)\Hto\Om(N_2)$ and $\Om(N_2)\Hto\Om(N_1)$.
By the observation above, since $\Omega(N_1)\subset N_1$ and $\Om(N_2)\subset N_2$, $N_1\Cto N_2$ and $N_2\Cto N_1$.
Hence, any point of \(N_1\) is chain-equivalent to any point of \(N_2\).
Since chain nodes are maximal chain-equivalent subsets of \(\cR_\cC\), this
implies $N_1=N_2$, contrary to the assumption that the nodes are distinct. 
Therefore, the assignment $N\longmapsto \operatorname{Node}_{{\cal C}_{H_\cF}}(\Om(N))$ must be injective on the set of chain nodes of $\cF$.
\end{proof}


\begin{theorem}
\label{thm: H one node => F one node}
    Assume that $\cF$ has compact dynamics.
    Then the chain graph of $H_\cF$ has at least as many nodes as the chain graph of $\cF$.
    In particular, if $H_\cF$ has a unique chain node, then so does $\cF$.
\end{theorem}

\begin{proof}
    By Proposition~\ref{prop: at least one node}, $\cC$ has at least one node.
    By Lemma~\ref{prop: N->Node(N) is injective}, if $H_\cF$ has a single node then $\cF$ cannot have more than one node.
\end{proof}

We show with an example that the inverse does not hold.
Let \(X=\{0,1\}\) with the discrete metric, and let \(\cF=\{f,g\}\), where $f(0)=0$, $f(1)=1$, $g(0)=1$, $g(1)=1$.
For sufficiently small \(\varepsilon\), chains are exact. The only chain node of \(\cF\) is \(\{1\}\), since \(0\) has \(1\) in its orbit but there is no chain from \(1\) back to \(0\).
On the other hand, $\mathcal K(X)=\{\{0\},\{1\},\{0,1\}\}$,
and the Hutchinson map satisfies
\[
H_\cF(\{0\})=\{0,1\},\qquad
H_\cF(\{1\})=\{1\},\qquad
H_\cF(\{0,1\})=\{0,1\}.
\]
Thus \(H_\cF\) has two chain nodes, namely the fixed points \(\{1\}\) and \(\{0,1\}\). 


\section{Example: the Levitt-Yoccoz gasket.}
\label{sec: cubic gasket}
Set $e_A=(0,0,1)$, $e_B=(1,0,1)$, $e_C=(1/2,\sqrt{3}/2,1)$ and consider the linear maps $\ell_P$, $P=A,B,C$, defined by 
$$
\ell_A(e_A)=e_A,\;
\ell_A(e_B)=e_A+e_B,\;
\ell_A(e_C)=e_A+e_C
$$
and similarly for $\ell_B,\ell_C$.
Let $X\subset\bR{\text{P}}^2$ be the closed triangle having vertices $A=[e_A], B=[e_B],C=[e_C]$ and containing the point $[e_A+e_B+e_C]$ and let $\cF=\{f_A,f_B,f_C\}$,
where $f_P$ is the projective automorphism induced by $\ell_P$.
By construction, $\ell_P(X)\subset X$ for $P=A,B,C$ and so $\cF$ is a projective IFS on $X$.
In the projective chart $z=1$, the vertices write $A=(0,0), B=(1,0),C=(1/2,\sqrt{3}/2)$
and the maps write
$$
f_A(x,y) = \left(\frac{x}{x+y/\sqrt{3}+1},\frac{y}{x+y/\sqrt{3}+1}\right),
$$
$$
f_B(x,y)=\left(\frac{y/\sqrt{3}+1}{-x+y/\sqrt{3}+2},\frac{y}{-x+y/\sqrt{3}+2}\right),
$$
$$
f_C(x,y)=\left(\frac{2y/\sqrt{3}-2x-1}{4(y/\sqrt{3}-1)},\frac{-\sqrt{3}}{4(y/\sqrt{3}-1)}\right).
$$
By construction, $f_P(P)=P$ for $P=A,B,C$ and a direct computation shows that each of these maps is a contraction at every point except at the fixed vertex.
Hence, this is not a contractive IFS.

\begin{figure}
    \centering
    \includegraphics[width=0.6\linewidth]{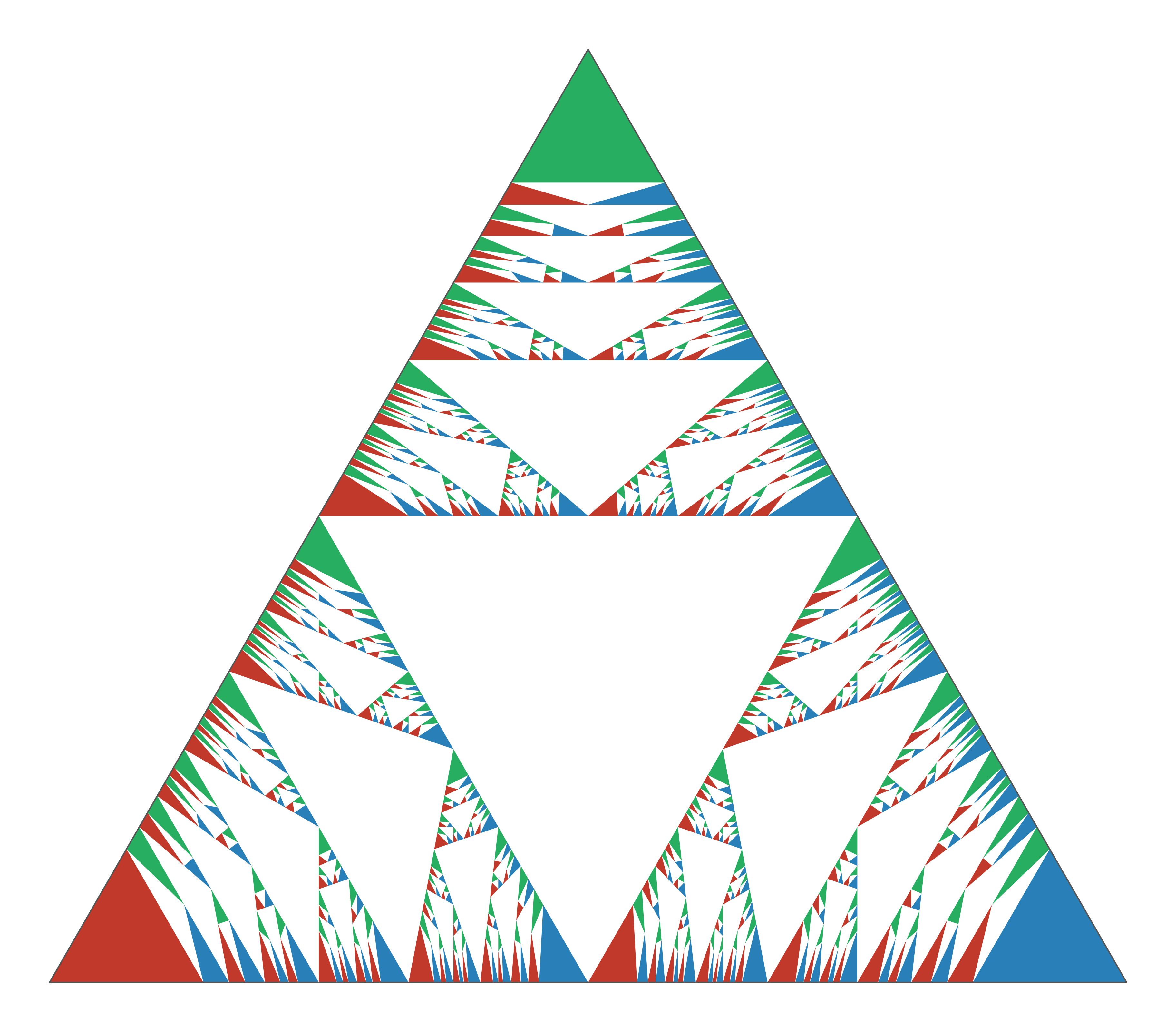} 
     \caption{{\BF The Levitt-Yoccoz gasket.}}
    \label{fig: logistic gaskets}
\end{figure}

This IFS was first introduced in literature by the present author in a 2008 preprint~\cite{DL08}, later refined in a 2012 preprint~\cite{DL12} whose main results were then published in~\cite{DL15} and~\cite{DL15b}.
It was later also studied by Arnoux and Starosta~\cite{AS13}, who refer to it as the ``Rauzy gasket''.
The invariant set of this IFS had already appeared in a 1993 article by Levitt about the dynamics of rotations pseudogroups~\cite{Lev93}, with a proof by Yoccoz that the measure of this set is zero.
This gasket is related to the topology of plane sections of the cubic polyhedron $\{4, 6|4\}$ (see~\cite{DD09} for details). 

Below we prove that $H_\cF$ has a single fixed point $K\in\KX$ and so that $\Gamma_{H_\cF}$ has a single node (and so no edge).
This node is precisely the fixed point $K$.
By Theorem~\ref{thm: H one node => F one node}, this also implies that $\cF$ has a single node, which is again the set $K$, this time seen as a subset of $X$.

The following standard consequence of the weakly contractive IFS theory is recalled for convenience.

\begin{proposition}
\label{thm: cubic gasket}
    Assume that $X$ is compact and each generator $f_k$ of $\cF$ is strictly distance-decreasing, namely
    $
    d(f_k(x),f_k(y))<d(x,y)$ for every $x,y\in X$ with $x\neq y$.
    Then:
    \begin{enumerate}
    \item $H_\cF$ is strictly distance-decreasing on $\KX$, i.e.
    $
    d_H(H_\cF(A),H_\cF(B))<d_H(A,B)
    $
    for all $A,B\in\KX$ with $A\neq B$;
\item $H_\cF$ has a unique fixed point $K_*\in \KX$;
\item for every $K\in \KX$ one has
$
H_\cF^k(K)\to K_*\text{ as }
k\to\infty.
$
\end{enumerate}
\end{proposition}

\begin{proof}
Fix two distinct compact sets $A,B\in \KX$ and set
$
h=d_H(A,B)>0.
$
For each $k=1,\dots,m$ consider the compact set
$
E_h=\{(x,y)\in X\times X:\ d(x,y)\le h\}.
$
The function
$
\varphi_k:E_h\to\mathbb R$, $\varphi_k(x,y)=d(f_k(x),f_k(y))$$,
$
is continuous. 
By hypothesis, for every $(x,y)\in E_h$ with $x\neq y$,
we have that $\varphi_k(x,y)<d(x,y)\le h$ and if $x=y$ then $\varphi_k(x,y)=0<h$. Hence $\varphi_k(x,y)<h$ for all $(x,y)\in E_h.$
Since $E_h$ is compact, $\varphi_k$ attains a maximum on $E_h$, so there exists $m_k<h$ such that
$
\varphi_k(x,y)\le m_k$ for all $(x,y)\in E_h.$

It follows that
$
d_H(f_k(A),f_k(B))\le m_k<h.
$
Indeed, let $a\in A$. 
Since $d_H(A,B)=h$, there exists $b\in B$ such that
$
d(a,b)\le h.
$
Hence $(a,b)\in E_h$, so
$
d(f_k(a),f_k(b))\le m_k.
$
Therefore every point of $f_k(A)$ lies within distance $m_k$ of $f_k(B)$.
By symmetry, every point of $f_k(B)$ lies within distance $m_k$ of $f_k(A)$.
Thus
$
d_H(f_k(A),f_k(B))\le m_k<h.
$

Now, $H_\cF(A)=\bigcup_{k=1}^m f_k(A)$ and $H_\cF(B)=\bigcup_{k=1}^m f_k(B)$,
so
$$
d_H(H_\cF(A),H_\cF(B))
\le
\max_{1\le k\le n} d_H(f_k(A),f_k(B))
\le
\max_{1\le k\le n} m_k
<h.
$$
Therefore, $d_H(H_\cF(A),H_\cF(B))<d_H(A,B)$ for all $A\neq B$, i.e.
$H$ is strictly distance-decreasing on $\KX$.

Since $X$ is compact, $\KX$ is compact as well. 
By Edelstein's fixed point theorem~\cite{Ede62}, a strictly distance-decreasing self-map of a compact metric space has a unique fixed point and every orbit converges to it. 
Therefore $H_\cF$ has a unique fixed point $K_*$ and $H_\cF^k(K)\to K_*$ for $k\to\infty$ for every $K\in \KX$.
\end{proof}

In order to use Proposition~\ref{thm: cubic gasket}, we need a criterion to show that the maps of this gasket are strictly distance decreasing.
The following result provides this criterion.

\begin{proposition}
    Let \(X\subset\mathbb R^n\) be compact and convex, and let \(f:X\to X\) be the restriction of a \(C^1\) map defined on a neighbourhood of \(X\). 
    Assume that $\|Df(x)\|\le 1$ for all $x\in X$ and that the equality set $E=\{x\in X:\|Df(x)\|=1\}$ is finite. 
    Then $d(f(x),f(y))<d(x,y)$ for $x\neq y$.
\end{proposition}

\begin{proof}
    Let \(x\neq y\) and let $\gamma(t)=(1-t)x+ty$, $0\le t\le1$, be the straight segment from \(x\) to \(y\). 
    Since \(X\) is convex, \(\gamma\subset X\). 
    Then 
    $$
    d(f(x),f(y))\le \ell(f\circ\gamma)=\int_0^1 \|Df(\gamma(t))(y-x)\|\,dt\le\|y-x\|\int_0^1 \|Df(\gamma(t))\|\,dt.
    $$
    Since the segment \(\gamma\) meets the finite set \(E\) in only finitely many points, it follows that $\|Df(\gamma(t))\|<1$ for almost every \(t\in[0,1]\), while \(\|Df(\gamma(t))\|\le1\) for all \(t\). 
    Therefore
    $\int_0^1 \|Df(\gamma(t))\|\,dt<1,$ and so $d(f(x),f(y))<d(x,y)$.
\end{proof}



\begin{theorem}
    \label{thm: LY}
    Let $\cF$ be the Levitt-Yoccoz gasket.
    Then:
    \begin{enumerate}
        \item $H_\cF$ has a unique fixed point $K\in\KX$;
        \item ${\cal G}_{H_\cF}={\cR}_{{\cal C}_{H_\cF}}=\{K\}$;
        \item $\Gl=\cR_\cC=K$;
        \item $K$ is chain-equivalent for $\cF$, i.e. the whole chain-recurrent set of $\cF$ is a single node.
    \end{enumerate}
    In particular, the graphs of $\cF$ and $H_\cF$ have a single node and no edge.
\end{theorem}

\begin{proof}
    Since $X$ is compact, $\cF$ has compact dynamics.
    The reader can verify that the Jacobian of the map $f_P$, with $P=A,B,C$, of this gasket has operator norm smaller than 1 at every point but $P$, where the norm is equal to 1.
    Hence, each of them is distance-decreasing and so, by Proposition~\ref{thm: cubic gasket}, (1) holds and $H^n_\cF(A)\to K$ for every $A\in\KX$.
    In turn, this implies that ${\cal G}_{H_\cF}=\{K\}$ and so (2) holds. 
    By Theorem~\ref{thm: H one node => F one node}, $\cF$ has at most one node.
    Denote by $N$ this node. 
    Since $K$ is the only compact non-empty $H$-invariant subset of $X$, the only possibility is that $\Om(N)=K\subset N\subset\Gl=K$, so $N=K$. Hence, (3) and (4) hold.
\end{proof}

\bibliographystyle{amsplain}  
\bibliography{refs}

\end{document}